\documentclass[reqno, 12pt]{amsart}

\usepackage{amsfonts, amsthm, amsmath, amssymb}

\usepackage{hyperref}
\usepackage{color}
\usepackage{helvet}

\RequirePackage{mathrsfs} \let\mathcal\mathscr

\numberwithin{equation}{section}

\newtheorem{theorem}{Theorem}[section]
\newtheorem{lemma}[theorem]{Lemma}
\newtheorem{proposition}[theorem]{Proposition}

\theoremstyle{definition}

\newcommand{\R}{\widehat{\mathbb R}^2}

\newcommand{\Bone}{B_{\mu,\beta_1,\beta_2}^{1}}
\newcommand{\bone}{B_{\mu,\beta_1,\beta_2}^{2}}

\newcommand{\dtwo}{\scrD_{\mu,\beta_1,\beta_2}^-}

\newcommand{\eone}{\scrE_{\mu,\beta_1,\beta_2}}
\newcommand{\etwo}{\scrF_{\mu,\beta_1,\beta_2}}
\newcommand{\cone}{\scrC^1_{\mu,\beta_1,\beta_2}}
\newcommand{\ctwo}{\scrC^2_{\mu,\beta_1,\beta_2}}

\newcommand{\sll}{\mathfrak sl(2,\mathbb R)}

\newcommand{\C}{\mathbb C}

\newcommand{\ve}{\varepsilon}
\def\CC{{\mathbb C}}

\def\NN{{\mathbb N}}

\def\RR{{\mathbb R}}

\def\ZZ{{\mathbb Z}}

\def\I{\operatorname{I}}

\def\scrB{{\mathcal B}}
\def\scrC{{\mathcal C}}
\def\scrD{{\mathcal D}}
\def\scrE{{\mathcal E}}
\def\scrF{{\mathcal F}}

\def\scrI{{\mathcal I}}

\def\scrO{{\mathcal O}}

\def\C{\operatorname{C{}}}

\def\L{\operatorname{L{}}}

\def\PGL{\operatorname{PGL}}

\def\SL{\operatorname{SL}}

\def\PSL{\operatorname{PSL}}

\def\vol{\operatorname{vol}}

\def\re{\operatorname{Re}}

\def\GamG{\Gamma\backslash G}

\def\PSLR{\PSL(2,\RR)}

\title{Uniform bounds for period integrals and sparse equidistribution}
\author{James Tanis}
\author{Pankaj Vishe}
\thanks{Tanis is partially supported by the French ANR grant GeoDyM (ANR-11-BS01-0004). Vishe is supported by the EPSRC Programme Grant: EP/J018260/1}
\begin{document}
\begin{abstract}
Let $M=\Gamma\backslash\PSLR$ be a compact manifold, and let $f\in \C^\infty(M)$ be a function of zero average. We use spectral 
methods to get uniform (i.e. independent of spectral gap) bounds for twisted averages of $f$ along long horocycle orbit 
segments. We apply this to obtain an equidistribution result for sparse subsets of horocycles on $M$.
\end{abstract}
\maketitle

\section{Introduction}

Let $\Gamma\subset \PSLR$ be a cocompact lattice, and let $M= \Gamma \backslash \PSLR$. Let $$
n(t) := \left(\begin{array}{cc} 1 & t \\ 0 & 1 \end{array}\right), \ \ a(t) := \left(\begin{array}{cc}e^{t/2} & 0 \\ 0 & e^{-t/2} \end{array}\right)\, ,
$$
denote the one-parameter subgroups generating the horocycle and the geodesic flows respectively.

Let $\L^2(M)$ be the space of complex-valued functions 
on $M$, which are square-integrable with 
respect to the $\PSL(2, \RR)$-invariant volume form. 
The space $\L^2(M)$ is a right regular representation of $\PSL(2, \RR)$ and 
any element of the Lie algebra $\sll$ of $\PSL(2, \RR)$ acts on $\L^2(M)$ 
as an essentially skew-adjiont differential operator. Let $\{Y, X, Z\}$ be a basis for the Lie algebra $\sll$ given by, 
$$
Y = \left(\begin{array}{rr}1/2 & 0 \\0 & -1/2\end{array}\right), \ \ X = \left(\begin{array}{rr}0 & 1 \\0 & 0\end{array}\right), \ \  Z = \left(\begin{array}{rr}0 & 0 \\1 & 0\end{array}\right).
$$ The center of the enveloping algebra for $\sll$ is generated by the Casimir element 
$$
\Box := - Y^2 - 1/2 (XZ + ZX)\,,
$$
which acts by multiplication by a constant on each irreducible, unitary representation of $\PSL(2, \RR)$. 
These constants, $\mu \in \text{spec}(\Box):=\RR_+\cup \{(-n^2+2n)/4:n\in\ZZ_+\}$ classify the nontrivial, unitary,
irreducible representations of $\PSLR$ into three categories:  
A representation is called principal series 
if $\mu \geq 1/4$,
 complementary series if $0 < \mu < 1/4$,
and discrete series if $\mu \leq 0$. 
The Casimir element takes the value zero on the trivial representation, 
which is spanned by the $\PSL(2, \RR)$-invariant volume form, denoted by $dg$.  

Let $\Omega_\Gamma$ be the set of eigenvalues of $\Box$ on $\L^2(M)$, 
counting multiplicities.   
Let
\begin{equation}\label{L^2-decomp}
\L^2(M)=\CC \bigoplus_{\mu \in \Omega_\Gamma} V_\mu, 
\end{equation}
be a Fourier decomposition of $\L^2(M)$, where for each $0<\mu $, $V_\mu$ is an irreducible, unitary representation 
of $M$ in the $\mu$-eigenspace of $\Box$; and for $\mu\leq 0$, $V_\mu=V_\mu^+\oplus V_\mu^-$ is a direct 
sum of two inequivalent irreducible unitary representations of $M$, called holomorphic and 
anti-holomorphic discrete series representations in the $\mu$-eigenspace of $\Box $. 

We will use the $\L^p$ based norms for $p = 1, 2, \infty$.  
We write $S_{p, 0}(f)$ for the $\L^p$ norm of $f$.  
Let $\mathcal O_s$ be the collection monomials in
$\{Y, X, Z\}$ up to order $s \in \ZZ_{\geq 0}$.  
Let $W^{s, p}(M)$ be the space of functions with bounded norm
$$
S_{p, s}(f) := \sum_{B \in \mathcal O_s} S_{p, 0}(B f)\,.
$$
The Hilbert Sobolev spaces $W^{s,2}(M)$  
are denoted by $W^{s}(M)$. For even integers $s$, they consist of functions $f$ with bounded norms
\begin{equation}\label{equa:Sobolev-elliptic}
S_{2, s}(f) = S_{2, 0}\left((I - Y^2 - 1/2X^2 - 1/2 Z^2)^{s/2} f \right)\,.  
\end{equation}
Using interpolation, these norms can be defined for all $s \geq 0$ (see \cite{L}).  
 \cite[Lemma 6.3]{N} implies that for all $s \in \ZZ_{\geq 0}$, 
there are constants $C_s, C_s' > 0$ such that 
\begin{equation}\notag
C_s  S_{2, s}(f) \leq \sum_{B \in \mathcal O_s} S_{2, 0}(B f) \leq C_s' S_{2, s}(f)\,.
\end{equation}
Moreover, $W^{s}(M)$ is endowed with an inner product, 
and by irreducibility, and (\ref{L^2-decomp}), we have 
\begin{equation}\label{equa:W^s-decomp}
W^s(M) = \CC\bigoplus_{\mu \in \Omega_\Gamma} V_\mu^s\,,
\end{equation}
where $V_\mu^s$ is the subspace $V_\mu \cap W^s(M)$. 
  
The distributional dual space of $W^{s}(M)$ is denoted by
$W^{-s}(M) := \left(W^{s}(M)\right)'$, equipped with the natural distributional norm $S_{2,-s}$. Our distributions are defined and studied in $W^{-s}(M)$.
The space of smooth functions on $M$ is denoted by
$C^\infty(M) := \cap_{s \geq 0} W^{s}(M)$, 
and its distributional dual space is $\mathcal E'(M) := \cup_{s\geq 0} W^{-s}(M)\,.$
For each $\mu \in \Omega_\Gamma$, 
let $V_\mu^\infty := \bigcap_{s \geq 0} V_\mu^s$.  

Our computations will be carried out in irreducible models 
consisting of functions defined on the real line.
Using unitary equivalence, we use the same notation $S_{2, s}(f)$ for $\L^2$-based norms 
in models.  We start by listing our results.
\subsection{Bounds for period integrals}
Let $\psi$ be the additive character 
$$
\psi(t) := e^{i a t},\textrm{  for all  }t \in \RR,
$$  
for some $a \in \RR\setminus \{0\}$.
For any $T \geq 1$ and for any $x\in M$, let $f\star \sigma_T(x)$ denote the unipotent period integral defined by
\begin{align}
 f\star\sigma_T(x):=\frac{1}{T} \int_0^T \psi(t)f(xn(t))dt\,.
\end{align}

Let $0 < \lambda_1 \leq 1/4$ be the spectral gap of 
the Laplacian on $M$, and let $\alpha_0 := \sqrt{1 - 4\lambda_1}$.  
Venkatesh \cite[Lemma 3.1]{V} used equidistribution and mixing of the horocycle flow to prove the following bound:
\begin{lemma}
\label{lem:ven}
 Let $f \in C^\infty(M)$ satisfying $\int_M f(g)\, dg=0$.  
 Then 
 $$
 S_{\infty,0}(f\star\sigma_T)\ll_\Gamma T^{-b}S_{\infty,1}(f),
 $$
 where $b<\frac{(1-2\alpha_0)^2}{8(3-2\alpha_0)}$, 
 and the implied constant is independent of $\psi$.
 
\end{lemma}

Our main theorem is an estimate of the $\L^\infty$ norm of the derivatives of $f\star \sigma_T $:

\begin{theorem}\label{theo:main theorem}
Let $f \in C^\infty(M)$ be such that $\int_M f(g)\,dg = 0$. Then  
for any $k\in\NN$, and for any $\ve > 0$, we have   
\begin{equation}
\label{eq:1/8}
S_{\infty,k}(f\star \sigma_T)\ll_{\Gamma,\ve} (1+|a|^{-1/2})
T^{2k-1/9 + \ve}S_{2,k+11/2+\ve}(f).
\end{equation}
Moreover, at the cost of a possibly larger, unspecified dependence on $a$, the above bound can be improved to get
\begin{equation}
\label{eq:1/8-}
S_{\infty,k}(f\star \sigma_T)\ll_{\Gamma,a,\ve}
T^{2k-(9-\sqrt{73})/4 + \ve}S_{2,k+11/2+\ve}(f),
\end{equation}
where $(9-\sqrt{73})/4< 1/8.772 $.
\end{theorem}
In particular, we remove the dependence of spectral gap in the exponent in Lemma \ref{lem:ven}, at the cost of a factor 
which depends on $\psi$. The dependence of $\Gamma$ in the constant can be made explicit using the injectivity radius and the spectral gap. We remark that one should not expect an estimate independent of 
both the spectral gap and $\psi$, since, as $a\rightarrow 0$, the behavior of $f\star\sigma_T$ is increasingly governed by 
the rate of equidistribution of the horocycle flow on $M$, which depends on the spectral gap. \\

\subsection{Sparse equidistribution results} Recently, there has been an increasing interest in studying 
equidistribution properties 
of sparse subsets of horocycle orbits. A question of Margulis \cite{M} asks whether the sequence 
$\{x_0 n(t_j)\}_{j \in \ZZ^+}$ is dense in $M$, 
for 
\begin{itemize}
\item $t_j$ \text{ is the }$j_{\text{th}}$\text{ prime number};
\item $t_j = \lfloor j^{1 + \gamma}\rfloor$, \text{ for some} $\gamma  > 0$.
\end{itemize}
A conjecture of Shah further predicts that the sequence $\{x_0 n(j^{1+\gamma})\}_{j \in \ZZ^+}$ is equidistributed for {\em any} $\gamma>0 $. Margulis' first question was partially answered by Sarnak-Ubis \cite{SU}, by proving that the horocycle orbit of a non-periodic point at prime 
values is 
dense in a set of positive measure in the modular surface.  

Margulis' second question was answered in \cite[Theorem 3.1]{V}, where lemma~\ref{lem:ven} was used to achieve 
equidistribution of the points $\{x_0n(j^{1 + \gamma})\} $, for any $0<\gamma<\gamma_{\text{max}}(\Gamma)$, depending 
on the spectral gap of $\Gamma$. We use theorem \ref{theo:main theorem} to remove the dependence of spectral gap in the 
above result, establishing equidistribution of points $\{x_0n(j^{1 + \gamma})\} $, for any $\gamma<1/26$, giving 
further evidence for Shah's conjecture.  

\begin{theorem}
 \label{thm:main theorem 2}
 Let $b<\frac{(1-2\alpha_0)^2}{8(3-2\alpha_0)}$, and let $b_1<1/9$,
 then for any $x_0\in M$, 
any $f\in C^\infty(M)$, and for any $0\leq \gamma<\frac{b+2b_1}{6-(b+2b_1)}$,
\begin{align*}
 \lim_{N\rightarrow\infty}\frac{1}{N}\sum_{j=1}^N f(x_0n(j^{1+\gamma})) = \int_M f(g)\,dg.
\end{align*}
In other words, the sequence $\{x_0n(j^{1+\gamma}):0\leq j\leq N\} $ 
is equidistributed in $M$ as $N$ tends to 
$\infty$. 
\end{theorem}
It should be noted that the methods in this paper in fact give {\em effective} bounds for the 
rate of equidistribution in the above theorem, which depend on $\gamma$. However, 
we do not to 
mention it here in order to keep the statement of the theorem more accessible. 

\subsection{Sparse equidistribution for smooth time-changes of horocycle flows} 
Let 
$\tau : M \times \RR \to \RR$ be a smooth cocycle over 
$\{n(t)\}_{t\in \RR}$, i.e. for all $x \in M$ and $t, s \in \RR$, 
$$
\tau(x, t + s) = \tau(x, t) + \tau(x n(t), s)\,.
$$ 
Assume that for all $x \in M$, 
$\tau(x, t)$ is a strictly increasing function of $t$.  
Let $\rho: M \to \RR^+$ denote the positive function defined by 
$$
\rho(x) = \frac{d}{dt} \tau(x, t)_{|_{t = 0}}\,.
$$
We assume $\rho \in W^6(M)$.  

For all $x \in M$, 
the smooth time change $\{n^\rho_t\}_{t\in \RR}$ 
of $\{n(t)\}$ is defined by 
$$
n^\rho_{\tau(x, t)}(x) := x n(t)\,.
$$
We will also write $n^\rho(\tau) := n^\rho_\tau$ .  
The vector field for $\{n^\rho_{\tau}\}_{\tau \in \RR}$ is generated by 
$$
X_\rho := X/\rho\,,
$$
and the $X_\rho$-invariant volume form is 
$
d_\rho g := \rho\, dg\,.
$  
In the wake of Shah's conjecture, 
it is natural to further ask whether 
Shah's conjecture holds for any smooth 
time change of a horocycle flow.

Using the method in \cite{V}, and the results of Forni-Ulcigrai \cite{FU}, 
we obtain a sparse equidistribution result for smooth 
time-changes of horocycle flow, dependent on the spectral gap, thus providing a partial answer to the above question.
\begin{theorem}\label{theo:equi-time-changemaps}
Let $b = -(1 - \alpha_0)^2/200$.  
Then for any $x_0 \in M$, any $f \in C(M)$, 
and any $0 \leq \gamma < b$, 
$$
 \lim_{N \to \infty} \frac{1}{N}\sum_{j = 1}^N f(x_0n^\rho(j^{1 + \gamma}))  = \int_M f(g) \,d_\rho g\,.
$$
\end{theorem}
\subsection{Remarks}
The method used here is simple yet powerful, and could be employed in answering further 
questions related to the 
horocycle flows. For instance, proposition \ref{prop:L^2_est} below gives a bound for the mean-square estimate for 
twisted averages of the horocycle flow, improving \cite[Theorem 4]{FU} in a very special case.
Moreover, since
analogous 
theory of Kirillov models is available for quotients of $\PGL(2,k)$, for a field $k$ in a very general setting, these 
estimates are likely to be generalized there as well. 

An independent work of the first author with L. Flaminio and G. Forni \cite{FFT} also addresses the question of bounding period integrals and application to Shah's conjecture. With more work, \cite{FFT} obtains stronger results by using a rescaling argument as in \cite{FF2} for the {\em twisted horocycle flow}, namely, a combination of the horocycle flow with a circle translation on $\Gamma \backslash \SL(2, \RR) \times S^1$.

Throughout the paper, we only deal with compact quotients of $\PSLR$. However, the non compact case can be dealt with 
analogously, using the corresponding spectral decomposition \cite[Theorem 1.7]{FF}. In this case, certain 
period integrals correspond to the Fourier coefficients of automorphic forms (see \cite[Section 3.2]{V}). A non-compact 
version of theorem \ref{theo:main theorem} would therefore provide uniform bounds for the Fourier coefficients of automorphic forms. 
Even though the methods in this paper would fail to give better estimates than those of Good \cite{G}, and 
Bernstein and Reznikov \cite{BR}; their advantage lies in their simplicity, and general applicability.

It should be noted that in order to keep the exposition clearer, we have not tried 
to optimize the Sobolev norms appearing in the results in this paper. These can be improved 
using a more stringent approach. 
\subsection{Acknowledgements} The second author would like to thank A. Venkatesh for introducing him to the 
problem, for the discussions, and for the encouragement. The authors would like to thank G. Forni for the encouragement, and 
for his generous help. The authors are also grateful for M. Baruch for his 
help in understanding the Kirillov model for $\PSLR$.

\section{Irreducible models and spectral decomposition}
\subsection{Line models}
\label{subs:irreducible_representations}

For a Casimir parameter $\mu > 0$, let $\nu = \sqrt{1 - 4\mu}$ 
be a representation parameter.  
The line model $H_\mu$ for a 
principal or complementary series representation space is realized on the Hilbert space consisting of 
functions on $\RR$ with the following norm.  
If $\mu \geq 1/4$, then $\nu \in i \mathbb{R}$, and the corresponding norm is 
$$S_{2, 0}(f) = \| f \|_{L^2(\mathbb{R})}.$$  
If $0 < \mu < 1/4$, then $0 < \nu < 1$, and 
$$
S_{2, 0}(f) = \left(\int_{\mathbb{R}^2} \frac{f(x) \overline{f(y)}}{|x - y|^{1 - \nu}} \,dx \,dy\right)^{1/2}\,.
$$  
The group action is defined by 
$$
\pi_\nu : \PSL(2, \mathbb{R}) \rightarrow \mathcal{U}(H_\mu)
$$ 
\begin{equation}\label{eq:line model action}
\pi_{\nu} (A) f(x) = |-c x + a|^{-(\nu + 1)} f\left(\frac{d x - b}{-c x + a}\right)\,,
\end{equation} 
where $A = \left(\begin{array}{rr}
a & b\\
c & d
\end{array}\right) \in \PSLR$, and $x\in \mathbb{R}$. 
Let $H_\mu^\infty$ be the space of all smooth vectors in $H_\mu$.

In the discrete series case the situation is a little bit more complicated. 
For $4\mu=-n^2+2n $, where $n\in \ZZ_{\geq 2}$, let $H^\infty_\mu$ be the space of smooth 
functions $f$ on $\RR$, such that $x^{-n}f(-1/x) $ is smooth.  
The corresponding group action $\pi_n$ is 
defined by

\begin{equation}\label{eq:line model action disc}
\pi_{n} (A) f(x) = |-c x + a|^{-n} f\left(\frac{d x - b}{-c x + a}\right)\,.
\end{equation} 
$H^\infty_\mu$ consists of two irreducible invariant 
subspaces for the action $\pi_n$, denoted by $H_\mu^{+,\infty}$, and $H_\mu^{-,\infty}$. 
These representation spaces correspond to the smooth vectors in holomorphic and 
anti-holomorphic discrete series representations $H_{\mu}^{+}$ and $H_\mu^-$, for the eigenvalue $\mu= (-n^2+2n)/4$, 
respectively. $H_\mu^\pm$ can be shown to be unitarily equivalent to $V^\pm_\mu$. See 
\cite[section 4]{B}, and \cite{T} for more details.

\subsection{Kirillov Models}\label{subs:Kirillov}
The Kirillov model, denoted by $K_\mu$, 
is closely related to the Fourier transform 
of the line model.  
\subsubsection{Principal and complementary series representations} 

For $\mu > 0$,  
we let 
\begin{equation}
 \phi : H_\mu \to  K_\mu : f \to C_\mu |x|^{(1 - \nu)/2} \hat f\,,
\label{eq: Kirillov map}
\end{equation}
where $C_\mu$ is a constant which is defined to be $1$ for $\mu\geq 1/4$, and it will be defined later for $0<\mu<1/4$.  
Then 
$$
K_\mu := \left\{\phi(f) : f \in H_\mu\right\}
$$
with the norm 
$$
S_{2, 0}(f) := \left(\int_{\RR / \{0\}} |f(x)|^2 \frac{dx}{|x|}\right)^{1/2}\,. 
$$
We begin by showing that $\phi$ is unitary in the complementary series case, the principal series case being simpler. For $0 < \mu < 1/4$, 
let $R(x)=|x|^{\nu-1}$ be a homogeneous function on 
$\RR\setminus \{0\}$. 
An easy computation shows that $\hat{R}(\xi)=|\xi|^{-\nu}\hat R(1)$. 
Moreover, $\hat R(1)$ is non zero since $\hat R$ is not identically zero. 
Then clearly
\begin{align*}
  \|f\|_{H_\mu}^2&=\int_{\RR^2}|x-y|^{\nu-1}f(x)\overline{f(y)}\,dx\,dy=\langle R*f,f\rangle\,,
\end{align*}
where $\langle\:\,,\:\rangle$ denotes the usual $L^2$ inner product on $\RR$. 
The Plancherel theorem implies
\begin{align*}
  \|f\|_{H_\mu}^2&=\int_{\RR\setminus\{0\}}|\hat f(\xi)|^2\hat R(\xi)\,d\xi=\hat R(1)\int_{\RR\setminus\{0\}}|\xi|^{-\nu}|\hat f(\xi)|^2\,d\xi\\
  &=\hat R(1)\int_{\RR\setminus\{0\}}\left||\xi|^{(1-\nu)/2}\hat f(\xi)\right|^2\frac{d\xi}{|\xi|}
=\int_{\RR\setminus\{0\}}\left|\phi (f)(\xi)\right|^2\frac{d\xi}{|\xi|}\,,
\end{align*}
upon choosing $C_\mu=|\hat R(1)|^{1/2}$, proving that $\phi$ is unitary. The action of $g \in \PSL(2, \RR)$ on $K_\mu$ is given by
$$
g \cdot \phi(f) := \phi(g \cdot f)\,,
$$
implying that $\phi$ is an unitary equivalence.  
The explicit action of $n(t)$, and $a(t)$, on $K_\mu$ is given by 
\begin{equation}\label{equa:action_kirillov}
n(t) \cdot f(x)  = e^{-i t x} f(t), \ \ \ \ a(t) \cdot f(x) = f(e^{t} x)\,.
\end{equation}
The explicit action of the basis $X,Y,Z$ of $\sll$ on this model is given by:
 \begin{equation}
 \label{eq:action_kirillov_Lie}
 X = - i x\,, \ \ \ \ 
 Y = x \frac{\partial}{\partial x}\,, \ \ \ \ 
Z =  i\frac{\mu}{x} - i x \frac{\partial^2}{\partial x^2}\,.
 \end{equation}
 
\subsubsection{Discrete series representation} 

A detailed description of these models can be found in \cite[Sections 4, 5]{M}. We will merely state the various 
results here. As before, for $4\mu=-n^2 + 2n$,  
we let 
$$
\phi : H_\mu\to  \L^2(\RR\setminus\{0\},dx/|x|) : f \to |x|^{1 - n/2} \hat f\,.
$$ 
\cite[(4.8)]{M} and \cite[(4.11)]{M} imply that $\phi$ maps $H_\mu^{+,\infty} $ into $\L^2((0,\infty), dx/|x|) $, and $H_\mu^{-,\infty} $ 
into $\L^2((-\infty,0),dx/|x|)$. Upon completion, this gives us the following explicit description of the Kirillov 
model for $V_\mu$:

The Kirillov model for a direct sum of holomorphic and anti-holomorphic discrete series representations of $\PSLR$, $K_\mu$, 
is also realized on the space 
$\L^2(\RR\setminus \{0\}, dx/|x| ) $. The action of the Borel subgroup, and the Lie algebra here is analogous to (\ref{equa:action_kirillov}), and (\ref{eq:action_kirillov_Lie}).

It should be noted that the Kirillov model the homolorphic and anti-holomorphic discrete series representations, is 
realized via $\phi$, on the spaces $\L^2((0,\infty),dx/|x|)$ and $\L^2((-\infty,0),dx/|x|)$ respectively. 

\subsubsection{Bounds for the elements in the Kirillov model} Throughout the paper, many of our computations in the 
space $V_\mu$ would be carried out in the Kirillov models $K_\mu$, via unitary equivalence. As noted before, these 
models respect the $\L^2$ based norms on $V_\mu$. However, using Sobolev embedding on $\RR$, we can get bounds for 
$\L^\infty$ norms on these models: 
\begin{lemma}\label{lemm:Sobolev-embedding-Kir}
For any Kirillov model of $\PSL(2, \RR)$,  
$$
\|f\|_{L^\infty(\RR)} \ll S_{2, 1}(f)\,.
$$
\end{lemma}
\begin{proof}
Let $f$ be a smooth function in a Kirillov model for $\PSL(2, \RR)$.  
For any $x \in \RR$, recall that $X f(x)=-ixf(x)$, and 
$Y f(x) = x f'(x).$  
Let $x_0 \in \RR$ and let $x_1 := \min\{3|x_0|/4,1/2\}$. 
Let $I_{x_0} := (x_0-x_1,x_0+x_1)$, and let $h$ be in 
$C_c^\infty(I_{x_0})$ be such that $h(x_0) = 1$, and $ \partial^k h(x)\ll |x_1|^{-k} $ for all $k\in \ZZ_{\geq 0} $.

We deal with the case $|x_0|<1$ first. 
By Sobolev inequality, 
$$
\begin{aligned}
|f(x_0)| \leq \|f h\|_{\L^\infty}  
&\ll  \|f h\|_{\L^1} + \|(f h)' \|_{\L^1}\notag\\
&\ll  \|f \|_{\L^1(I_{x_0})} + \| f' \|_{\L^1(I_{x_0})}+|x_0|^{-1}\| f \|_{\L^1(I_{x_0})}\notag\\
&\ll \|f \|_{\L^1(I_{x_0})} + |x_0|^{-1}\| Yf \|_{\L^1(I_{x_0})}+|x_0|^{-1}\| f \|_{\L^1(I_{x_0})}\notag\,.
\end{aligned}
$$
We consider the term $\|f \|_{\L^1(I_{x_0})}$:
$$
\begin{aligned}
 \|f \|_{\L^1(I_{x_0})}&=\int_{x_0-x_1}^{x_0+x_1}|f(x)|dx\\
 &\ll |x_0|^{1/2}\int_{x_0-x_1}^{x_0+x_1}|x|^{-1/2}|f(x)|dx\\
 &\ll |x_0|^{1/2}\left(\int_{x_0-x_1}^{x_0+x_1}|x|^{-1}|f(x)|^2dx\right)^{1/2}x_1^{1/2}\\
 &\ll |x_0|\|f \|_{\L^2(I_{x_0},dx/|x|)}\,.
\end{aligned}
$$
After analogously bounding the rest of the terms, 
we get
$$
\begin{aligned}
\|f h\|_{L^\infty}
&\ll |x_0|\|f \|_{\L^2(I_{x_0},dx/|x|)} +\|Yf \|_{\L^2(I_{x_0},dx/|x|)}+\| f \|_{\L^2(I_{x_0},dx/|x|)}\notag\\
&\ll  S_{2,1}(f)\,.\label{eq:sobobound}
\end{aligned}
$$
When $|x_0| > 1$, the
Sobolev embedding implies
$$
\begin{aligned}\label{eq:sobobound1}
\|f h\|_{L^\infty}&\ll  \|f h\|_{\L^1} + \|(f h)' \|_{\L^1}\ll  \|f \|_{\L^1(I_{x_0})} + \| f' \|_{\L^1(I_{x_0})}\notag\\
&\ll \int_{I_{x_0}} |xf(x)||x|^{-1}dx +\int_{I_{x_0}} |xf'_\mu(x)||x|^{-1}dx \notag \\
&\ll S_{2,1}(f)\,,\notag 
\end{aligned}
$$
proving the lemma.
\end{proof} 

\subsection{Invariant distributions for the horocycle flow}
Let $\sigma_{pp}$ be the set of non negative eigenvalues of the Laplace-Beltrami operator on $M$, which coincides with the 
set of non negative eigenvalues of the Casimir operator.   
The distributions invariant under the horocycle flow 
have been classified by Flaminio-Forni \cite{FF}. 
They showed that this space, 
$\mathcal I(M)$, has an infinite countable dimension, 
and that there is a decomposition 
$$
\mathcal I(M) = \bigoplus_{\mu \in \sigma_{pp}} \mathcal I_\mu \oplus \bigoplus_{n \in \ZZ^+} \mathcal I_n \,,
$$
where 
\begin{itemize} 
\item  for $\mu = 0$, the space $\mathcal I_0$ is spanned by scalar multiples of 
$\PSL(2, \RR)$-invariant volume, denoted by $\mathrm{vol}$; 
\item  for $0 < \mu < 1/4$, there is a splitting 
$\mathcal I_\mu = \mathcal I_\mu^+ \oplus \mathcal I_\mu^-$, 
where $\mathcal I_\mu^{\pm} \subset W^{-s}(M)$ 
if and only if $s > \frac{1 \pm \sqrt{1 - 4 \mu}}{2}$, 
and each subspace has dimension equal to the multiplicity of $\mu \in \sigma_{pp}$; 
\item for $\mu \geq 1/4$, the space 
$\mathcal I_\mu \subset W^{-s}(M)$ if and only if $s > 1/2$, 
and it has dimension equal to twice the multiplicity of $\mu \in \sigma_{pp}$;
\item for $n \in \mathbb Z_{\geq 2}$, the space $\mathcal I_n \subset W^{-s}(M)$ 
if and only if $s > n /2$ and it has dimension equal to twice the rank of the space of holomorphic sections of the 
$n_{\text{th}}$ power of the canonical line bundle over $M$.
\end{itemize} 

For $s > 1/2$, let $\mathcal I_\mu^s := \mathcal I_\mu \cap W^{-s}(M)$ 
and $\mathcal I_n^s := \mathcal I_n \cap W^{-s}(M)$.  
By \cite[Theorem 1.4]{FF}, for all $\mu \neq \frac{1}{4}$, and $n \in \mathbb Z_{\geq 2}$, 
the spaces $\mathcal I_\mu^s$ and $\mathcal I_n^s$ 
have a basis of unit-normed (in $W^{-s}(M)$) eigenvectors 
for $\{a(t)\}_{t\in \RR}$, which we denote by $\mathcal B_{\mu}^s$ and $\mathcal B_n^s$, respectively.  
The space $\mathcal I_{1/4}^s$ decomposes as 
$\mathcal I_{1/4}^{s, +} \cup \mathcal I_{1/4}^{s, -}$,
where $\mathcal I_{1/4}^{s, -}$ has a basis of unit-normed 
eigenvectors for $\{a(t)\}_{t\in \RR}$ 
denoted by $\mathcal B_{1/4}^{s, -}$, and $\mathcal I_{1/4}^{s, +}$ 
has a basis of unit-normed 
generalized eigenvectors, 
denoted by $\mathcal B_{1/4}^{s, +}$.  
Let 
$$
\mathcal B_{0, +}^s := \bigcup_{\mu \in \sigma_{pp}} B^s_\mu \cup \left\{\mathcal B_{n}^s : n = 2\right\}\,,
$$
and 
$$
\mathcal B_+^s := \bigcup_{\mu \in \sigma_{pp} / \{0\}} \mathcal B_\mu^s\,.
$$

\subsection{Spectral decomposition for averages of horocycle flow} 
Let $x_0$ be a fixed arbitrary point of $M$. Then for any $T \geq 1$, let $\nu_T$ be defined on $L^2(M)$ by 
 $$
 \nu_T(f) := \frac{1}{T} \int_0^T  f(x_0n(t)) \,dt\,.
 $$
 For $\mu \in \sigma_{pp}$, and for $\mathcal D \in \mathcal I_{\mu}^{\pm}$,  
 let $S_{\mathcal D} := \frac{1 \pm \re\sqrt{1 - 4 \mu}}{2}.$  
 
For any $s > 2$, we may project $\nu_T$ orthogonally in 
 $W^{-s}(M)$ onto the basis $\mathcal B^{s}(M)$ 
and the orthogonal complement of its closed linear span, $\mathcal I^{s}(M)^\bot$.  
Then for all $\mu \in \Omega_{\Gamma}$, 
there exists distributions 
$$
\mathcal D := 
\left\{
\begin{array}{ll}
\mathcal D_{x_0, T, \mu}^{\pm} \in \mathcal I_{\mu}^{\pm},  & \text{ if } \mu \in \sigma_{pp}\,, \\
\mathcal D_{x_0, T, \mu} \in  \mathcal I_n,  & \text{ if } 4\mu = -n^2 + 2n \text{ for } n \in \ZZ_{\geq 2}\,,
\end{array}
\right.
$$
and  
$\mathcal R^s_{x_0, T} \in W^{-s}(M)$ is such that 
\begin{align}\label{eq:FF}
 \nu_T = \left(\vol + \sum_{\mathcal D} 
\mathcal D\right) 
 \oplus\frac{\mathcal R^s_{x_0,T}}{T}  
 \end{align}
 in the $W^{-s}(M)$ Sobolev structure.  

\cite{FF} showed that for any $s' > 3$, $S_{2,-s'}(\mathcal R^{s'}(x_0, T))\ll_s 1$.
We now use    
arguments in \cite[Section 5]{FF} and \cite[Lemma 3.7]{FF2}, to prove that $\mathcal R^{s'}(x_0, T)\in 
W^{-s'}(M)$, for any $s'>2$, along with a suitable bound for this norm. Being able to estimate the $W^{2+\ve}$ norm of remainder distribution would enable us to get a stronger 
decay estimate in theorem 
\ref{theo:main theorem}.  

Henceforth we assume $2 < s \leq 3$, and let $s_{\mathrm{reg}}:=\sup_{\ve>0}\{\lfloor 2s-\ve\rfloor$\}.
For $n \in \ZZ_{s_{\mathrm{reg}}}$, 
the above description of invariant distributions shows 
$\mathcal I_n(M) \not\subset W^{-s}(M)$, implying that these distributions do not appear in the 
decomposition (\ref{eq:FF}) of $\nu_T$ in $W^{-s}(M)$.  
Therefore, using the definition of $\mathcal R_{x_0, T}^{s}$, we get 
\begin{equation}\label{equa:remainder}
\mathcal R_{x_0, T}^{s} \in \bigoplus_{n = 3}^{s_{\mathrm{reg}}} \mathcal I_n(M) \oplus \mathcal I^{s}(M)^\bot\,.
\end{equation}

\begin{lemma}\label{lemm:remainder-est}
Let $2 < s \leq 3$.  
Let $\mathcal R_{x_0, T}^{s}$ be as in (\ref{eq:FF}).  
Then for all $f \in C^\infty(M)$, 
$$
|\mathcal R^{s}_{x_0,T} (f)| \ll_{s} \frac{1}{\sqrt{1 - \sqrt{1 - \lambda_1}}}S_{2, s}(f)\,.
$$
\end{lemma}
\begin{proof}
We begin by observing that for any $n\in\ZZ_{\geq 2}$, for any $s>n/2$, and for any $f\in W^s(M)$, $
\mathcal R_{x_0, T}^s\mid_{\scrI_n^s}(f)=T \nu_T\mid_{\scrI_n^s}(f)$ . \cite[Lemma 5.12]{FF} further implies that
$$
|T\nu_T\mid_{\mathcal I_n^s}(f)| \ll_{s} S_{2, s}(f).
$$ 
These bounds clearly suffice for any $2<s\leq 3$, and for any $n=3,4,5$.

It is therefore enough to consider $f$ to be a function on which each $\scrD\in 
\scrI^s(M)$ vanishes. For such a function $f$, $\int_0^T f(x_0n(t)) dt=\mathcal R^s_{x_0, T}(f) $. For any $s > 2$, 
 \cite[Theorem 4.1]{FF} imples the existence of a function $g \in C^\infty(M)$ 
(unique up to additive constants) satisfying
\begin{equation}\label{equa:X-existence}
X g = f,
\end{equation}
such that for any $0 \leq t < s - 1$, 
\begin{equation}\label{equa:transfer-sobolev}
S_{2,t}(g) \ll_{t, s} \frac{1}{\sqrt{1 - \sqrt{1 - \lambda_1}}} S_{2,s}( f)\,.
\end{equation}
The fundamental theorem of calculus then implies   
$$
|\mathcal R^s_{x_0, T}(f)| = |\int_0^T X g (x_0n(t)) dt|
= | g (x_0n(T)) - g(x_0)|.$$

Now we estimate each term on the right-hand side.  
Let $\tau_0 \in \{0, T\}$, and let $x_{\tau_0} := x_0 n(\tau_0)$.  
As in the proof of \cite[ Lemma 3.7]{FF2}, 
the mean value theorem implies that 
$$
\int_0^1  g(x_{\tau_0}n(\tau)) d\tau =  g(x_{\tau_0}n(\tau_1))\,,
$$
for some $\tau_1 \in (0, 1)$.
A further application of the fundamental theorem of calculus, and (\ref{equa:X-existence}) gives us 
\begin{align}\label{equa:trace-uniform}
   \int_0^1  g(x_{\tau_0}n(\tau))\, d\tau+\int_{\tau_1}^0  f(x_{\tau_0}n(\tau))\, d\tau 
=g(x_{\tau_0}n(\tau_1))+\int_{\tau_1}^0 Xg(x_{\tau_0}n(\tau)) dt =g(x_{\tau_0}) 
\end{align}

We now apply the Sobolev trace theorem in \cite[Lemma 3.7]{FF2} to the operator which maps $g$ to the trace
$\int_0^1g(x_{t_0}n(t))dt$, to get that for any $\ve>0$, 
\begin{align}\label{equa:trace-transfer}
|\int_0^1 g(x_{\tau_0}n(t))\,dt| & \ll_\ve S_{2, 0}\left((I- Z^2- Y^2)^{1/2+\ve} g\right) \ll_\ve S_{2, 
1 + \ve}(g)\ll_{\ve, \lambda_1} S_{2,2+2\ve} (f).
\end{align}
The trace theorem also analogously implies that
$$
|\int_{\tau_1}^0 f(x_{\tau_0}n(\tau))\, d\tau| \ll_{\ve, \lambda_1} S_{2,  2 + 2\ve} (f)\,.  
$$
Combining these bounds, we get
$
|g(x_{\tau_0})| \ll_{\ve, \lambda_1} S_{2,  2 + 2\ve} (f)\,,
$
thus implying
$$
|\mathcal R^s_{x_0, T}(f)| \ll_{\ve, \lambda_1} S_{2,  2 + 2\ve} (f)\,,
$$
where the implied dependence on $\lambda_1$ is 
$C_{\lambda_1} = \frac{1}{\sqrt{1 - \sqrt{1 - \lambda_1}}}$, thus proving the lemma for any $s>2$, upon suitably 
choosing $\ve$.
\end{proof}

\subsubsection{Explicit spectral decomposition for $\nu_T$} In light of the improved regularity of the remainder 
distribution in lemma \ref{lemm:remainder-est}, 
 \cite[Theorem 1.4]{FF} implies that
 for any $s>2$ and $(x,T)\in M\times \RR_+$,  there exists distributions 
 $\mathcal D^{s}_{2, x_0, T} \in \mathcal B_{0, +}^s / \mathcal B_+^s$ and  
$\mathcal R^s_{x_0, T} \in W^{-s}(M)$  
and a sequence of real-valued functions 
$\{c_{\mathcal D}^s(\cdot, \cdot)\}_{\mathcal D \in \mathcal B^s}$ on $M \times \RR_{\geq 1}$
such that for all $f \in C^\infty(M)$, 
\begin{align}\label{eq:FF3}
 \nu_T(f) & = \int_M f(g)\, dg+ \sum_{\mathcal D\in \mathcal B_{+}^s /
\mathcal B_{1/4}^{s, +}} 
c_{\mathcal D}^s(x_0,T)\mathcal D(f) T^{-S_{\mathcal D}}  + \sum_{\mathcal D \in \mathcal B_{1/4}^{s, +}} 
c^s_{\mathcal 
D}(x_0,T)
\mathcal D(f) T^{-\frac{1}{2}}\log T \notag \\
 & \ \ \ \ \ \ \ \ \ \ \ \ \ \ \ \ \ \ \ \ \ \ \ \ \ \ \ \ \ \ \ \ \ \ \ \ \ \ \ + \frac{\mathcal D^{s}_{2, x_0,T}(f) 
\log T + \mathcal R^s_{x_0,T}(f)}{T}  \,,
 \end{align}
 where, 
\begin{equation*}
\sum_{\mathcal D\in \scrB_{+}^{s}} |c_{\mathcal D}^s(x_0,T)|^2 
+ S_{2,-s}(\mathcal D_2^s(x_0,T))+ \ll_s 1\, ,\,\,\,\, S_{2,-s}(\mathcal R^s_{x_0,T})\ll_{\lambda_1,s} 1
\end{equation*}
using \cite[Corollary 5.3]{FF} and lemma \ref{lemm:remainder-est}.

Note that for any $0<\mu<1/4$, and for any $\scrD\in \scrI_\mu^{s,\pm}\cap \mathcal B_{+}^s$, the corresponding value of $S_{\scrD}=(1\pm\nu)/2$. Thus the contribution in \eqref{eq:FF3} from terms corresponding to $\scrD\in \scrI_\mu^{s,+}\cap \mathcal B_{+}^s$ is at most $O(T^{-1/2})$, for any $s>1$. To summarize,
for any $\ve>0$ we have
\begin{align}\notag
 \nu_T(f) & = \int_{M} f dg+ \sum_{\mu\in\Omega_\Gamma\cap(0,1/4)} 
c^{2+\ve}_{\mathcal D_\mu^-}(x_0,T)\mathcal D_\mu^-(f)T^{(-1+\nu)/2} + O_s(S_{2,1+\ve}(f)T^{-\frac{1}{2}}\log^+T) \\
 & \ \ \ \ \ \ \ \ \ \ \  \ \ \ \ \ \ \ \ \ \ \ \ \ \ \  \ \ \ \ \ \ \ \ \ \ \ \ \  +
O_s(S_{2,2+\ve}(f)T^{-1}\log^+T)\, ,\label{eq:FF1}
\end{align}
where for every $\mu\in\Omega_\Gamma\cap(0,1/4) $, $\scrD^-_\mu\in\scrI_\mu^-$, $S_{2,-1-\ve}(\scrD^-_\mu)\ll_s 1 $, and $|c^{2+\ve}_{\mathcal D_\mu^-}(x_0,T)|\ll_s 1 $.

\section{Proof of Theorem \ref{theo:main theorem}}
We start by recalling that for any $f\in C^\infty(M)$,
\begin{equation}\notag
f\star \sigma_T(x)=\frac{1}{T}\int_0^T \psi(t)f(xn(t))dt\,.
\end{equation}
For any $f \in 
C^\infty(M)$, the following can be easily verified: 
\begin{align}
X(n(t)\cdot f )&=n(t)\cdot (X f)\notag\\
Y ( n(t) \cdot f) & = n(t) \cdot ((Y + t X) f) \notag \\
Z (n(t) \cdot f) & = n(t) \cdot ((Z - 2 t Y - t^2 X) f).\label{eq:deri}
\end{align} 
These bounds imply that for any $s\in \ZZ_+$, any degree $s$ monomial $B_0$ in $X,Y,Z$, and any $x\in M$,
\begin{align}
\label{eq:infty 2sbound}
|B_0(f \star \sigma_T)(x)|\ll\sum_{j=0}^{s}\sum_{B\in\scrO_j}\sum_{k=0}^{2j}|Bf\star\sigma_T^k(x)|,
\end{align}
where
\begin{equation}\label{eq: k-sigma def}
f\star \sigma^k_T(x):=\frac{1}{T}\int_0^T \psi(t)t^kf(xn(t))dt\,.
\end{equation}
Note that $f\star \sigma_T^0 $ is equal to $f\star \sigma_T$. It is therefore enough to obtain a suitable bound for $S_{\infty,0}(Bf\star\sigma_H^k)$, for $k\in \ZZ_+$. Let $x_0\in 
M$ be a fixed arbitrary point. We begin by noting
\begin{align}
\notag ((f\star\sigma_H)&\star\sigma^k_T)(x_0)\\ \notag&=\frac{1}{TH}\int_0^T\int_{0}^H 
\psi(t+z)t^kf(x_0n(t+z))\,dz\,dt\\ \notag &=\frac{1}{TH}\int_0^{T+H}\int_{0}^{\min\{y,H\}} 
\psi(y)(y-z)^kf(x_0n(y))\,dz\,dy\\ \notag &=\frac{1}{TH}\int_H^T\int_{0}^{H} 
\psi(y)y^kf(x_0n(y))\,dz\,dy+O(HT^{k-1}S_{\infty,0}(f) )\\ \notag &=f\star\sigma_T^k(x_0)+O(HT^{k-1}S_{\infty,0}(f) ),
\end{align}
 implying $S_{\infty,0}(f\star\sigma^k_T-(f\star\sigma_H)\star\sigma^k_T)\ll HT^{k-1}S_{\infty,0}(f)  $. An 
application of the Cauchy-Schwarz inequality implies 
\begin{align}\label{equa:CS-v_T}
|((f\star\sigma_H)\star\sigma^k_T)(x_0)|^2\ll T^{2k}\nu_T(|f\star\sigma_H|^2)\,.
\end{align}
In the light of \eqref{eq:infty 2sbound}, this implies that for any $B_0\in \scrO_s$,
\begin{align}
 |B_0(f \star \sigma_T)(x_0)| \ll 
\sum_{j=0}^{s}\sum_{B\in\scrO_j}\nu_T(|Bf\star\sigma_H|^2)^{1/2}T^{2j}+HT^{2s-1}S_{\infty,s}(f).\label{eq:inftybound}
\end{align}
In order to bound the term $\nu_T(|Bf\star\sigma_H|^2)$, it is enough to get an appropriate bound for 
$\nu_T(|f\star\sigma_H|^2)$.
We begin by noting that for $f\in \L^2(M)$, 
we have $|f\star\sigma_H|^2 \in C^\infty(M)$. 
We use the spectral decomposition in \eqref{eq:FF1} to get that for any $\ve>0$,
\begin{align}\label{eq:FFfstar}
 \nu_T&(|f\star\sigma_H|^2) = \int_{M} |f\star\sigma_H|^2 dg\\ \notag&+ \sum_{\mu\in\Omega_\Gamma\cap(0,1/4)} 
c_{\mathcal D_\mu^-}(x_0,T)\mathcal D_\mu^-(|f\star\sigma_H|^2)T^{(-1+\nu)/2} + 
S_{2,1+\ve}(|f\star\sigma_H|^2)T^{-\frac{1}{2}}\log^+T \\
 & \ \ \ \ \ \ \ \ \ \ \  \ \ \ \ \ \ \ \ \ \ \ \ \ \ \  \ \ \ \ \ \ \ \ \ \ \ \ \  +
S_{2,2+\ve}(|f\star\sigma_H|^2)T^{-1}\log^+T\,,\notag
\end{align}
using the Fourier expansion (\ref{L^2-decomp})
$$
|f\star\sigma_H|^2 = \bigoplus_{\mu \in \Omega_\Gamma} (|f\star\sigma_H|^2)_\mu\,.
$$

\begin{lemma}\label{lemm:nu_T}
For any $\ve > 0$, we have
$$
\begin{aligned}
|\nu_T(|f\star\sigma_H|^2)|  &\ll_{\ve} \int_{M} |f \star \sigma_H|^2 dg  + T^{-\alpha_0}
\sum_{\mu\in (0,1/4)\cap 
\Omega_\Gamma}c_\mu(x_0,T) \mathcal D_\mu^-(|f\star\sigma_H|^2)\notag \\
&+S_{2,1+\ve}(|f\star\sigma_H|^2) T^{-\frac{1}{2}}\log^+T+ S_{2,2+\ve}(|f\star \sigma_H|^2)T^{-1} \log^+(T)\,,
 \end{aligned}
$$
where $\sum_{\mu}|c_\mu(x_0,T)|^2\ll_\ve 1 $. 
\end{lemma}

It remains to estimate the terms on the right-hand side of lemma \ref{lemm:nu_T}.

\subsection{Estimate of $\int_{M} |f \star \sigma_H|^2\, dg$}
In order to bound $\int_{M} |f \star \sigma_H|^2\, dg=S_{2,0}(f \star \sigma_H)^2$, we start by obtaining a slightly 
more general bound $S_{2,s}(f \star \sigma_H)^2 $, the utility of which will be evident in the later part of this section. 
Using the explicit action of the Lie algebra (\ref{eq:deri}), a bound similar to (\ref{eq:infty 2sbound}) can be 
obtained for any $s\in \NN$:
\begin{align}
\label{eq:2sbound}
 S_{2,s}(f \star \sigma_H)^2\ll\sum_{j=0}^{s}\sum_{B\in\scrO_j}\sum_{k=0}^{2j}S_{2,0}(Bf\star\sigma_H^k)^2.
\end{align}

For any $f \in L^2(M)$, and any $\mu\in\sigma_{pp}\setminus\{0\}$, let $f_\mu$ be the projection 
of $f$ in the component $V_\mu$. Using the fact that the operators $\star\,\sigma_H^k$ are limits of discrete sums of 
operators of type $t^kn(t)\cdot f$, which map $V_\mu$ into $V_\mu$, for any $\mu\in\Omega_\Gamma$, we can easily see 
that the operation $\star\,\sigma_H^k$ splits across irreducible components, i.e.
\begin{align}
\label{eq:split}
 (f\star\sigma_H^k)_\mu=f_\mu\star\sigma_H^k.
\end{align}

\begin{proposition}
 \label{prop:L^2_est}
Let $f\in C^\infty(M)$ be such that $\int f(g)\,dg=0 $.  Then for any $s\in \RR_+$,
$$
S_{2,s}(f\star \sigma_H)^2 \ll (1+ |a|^{-1}) H^{4s-1}  S_{2,s+ 1}(f)^2\,.
$$
\end{proposition}
\begin{proof}
Write 
$$
f = \bigoplus_{\mu \in \Omega_\Gamma} f_\mu\,,
$$
where $f_\mu \in V_\mu$ for all $\mu \in \Omega_\Gamma$.  
By \eqref{eq:split}, we have
$$
f \star \sigma_H^k = \bigoplus_{\mu \in \Omega_\Gamma} (f_\mu \star \sigma_H^k)\,.
$$
In the light of \eqref{eq:2sbound}, it is enough to get the corresponding bound for $S_{2,0}(f\star \sigma_H^k)^2$.
Parseval's identity implies that, 
\begin{equation}\label{equa:fsigma_decomposition}
S_{2,0}(f\star\sigma_H^k)^2=\sum_{\mu \in \Omega_\Gamma} S_{2,0}(f_\mu\star 
\sigma_H^k)^2\,. \notag
\end{equation}
For any $\mu \in \Omega_\Gamma$, we start estimating $S_{2,0}(f_\mu\star\sigma_H^k)^2$ in the Kirillov 
model for $V_\mu$. For $f_\mu\in 
K_\mu$,  the explicit action of $\star\,\sigma_H^k$ is given by
\begin{align}
 f\star\sigma_H^k(x)=\frac 1H\int_0^H t^ke^{i(a-x)t}\,dt\, f(x).
\end{align}
Let $C \geq 3$, then if $|H a| \leq C$,
then we obtain the trivial bound 
\begin{align*}
\int_{\RR/\{0\}} |f_\mu\star \sigma^k_H(x)|^2 \frac{dx}{|x|} & \ll H^{2k} S_{2, 0}(f_\mu)^2.
\end{align*}    
Henceforth, we assume $|H a| \geq C$. Let $I_1=(a-1/H,a+1/H)$, $I_2=(-1/H,1/H)\setminus\{0\}$, and $I_3=\RR\setminus 
(I_1\cup I_2\cup\{0\})$. Then the trivial bound
\begin{align}
 &\int_{I_1}\left|\frac 1H\int_0^H t^ke^{i(a-x)t}\,dt\, 
f_\mu(x)\right|^2\frac{dx}{|x|}\ll H^{2k-1}\|f_\mu\|_{\L^\infty(\RR)}^2(1+a^{-1}),\notag
\end{align}
is enough upon further using lemma \ref{lemm:Sobolev-embedding-Kir}.
On $I_2\cup I_3 $, using repeated integration by parts, we get
\begin{align}
 &\int_{I_2\cup I_3}\left|\frac 1H\int_0^H t^ke^{i(a-x)t}\,dt\, 
f_\mu(x)\right|^2\frac{dx}{|x|}\notag\\
 &\ll \int_{\I_2\cup I_3}\left|\frac{1 - e^{-i (a - x)H}}{i H(a - 
x)^{k+1}} 
f_\mu(x)\right|^2\frac{dx}{|x|} +\sum_{j=1}^{k}H^{2j}\int_{I_2\cup I_3}\left|\frac{ 
f_\mu(x)}{H(x-a)^{k+1-j}}\right|^2\frac{dx}{|x|}\notag\\
&\ll H^{2k-2}\int_{\I_2\cup I_3}\frac{|f_\mu|^2\,dx}{|x(x-a)^2|}\notag\\
&\ll H^{2k-1}(1+a^{-1})\int_{\I_2}\frac{|f_\mu|^2\,dx}{|x|}+H^{2k-2}(1+a^{-1})\|f_\mu\|_{\L^\infty(\RR)}^2\int_{I_3}(|x|^{-2}+|x-a|^{-2})\,dx\notag\\
&\ll H^{2k-1}(1+a^{-1})S_{2,0}(f_\mu)^2+(1+a^{-1})H^{2k-1}S_{2,1}(f_\mu)^2.\notag
\end{align}
 Combining these bounds, we get
\begin{align*}
 S_{2,0}(f_\mu\star \sigma_H^k)^2\ll (1 + |a|^{-1}) H^{2k-1}S_{2, 
1}(f)^2.
\end{align*}
Using these bounds, along with \eqref{eq:2sbound}, we get that for any $s\in \NN$,
\begin{align*}
 S_{2,s}(f_\mu\star \sigma_H)^2\ll (1 + |a|^{-1}) H^{4s-1}S_{2, 
s+1}(f)^2.
\end{align*}
Upon interpolation, we prove the above bound for any $s\in \RR_+$, and upon further
adding over all $\mu\in \Omega_\Gamma$ we get the proposition.
\end{proof}

\subsection{Estimate of $\mathcal D^{-}_\mu(|f \star \sigma_H|^2)$}

In this subsection, we will estimate 
$\mathcal D^{-}_\mu(|f \star \sigma_H|^2)$ for $0<\mu<1/4$. Recall that $\nu=\sqrt{1-4\mu}$.
Our main goal will be 
to establish the following proposition:
\begin{proposition}\label{prop:distribution}
Let $f \in \C^\infty(\GamG)$ be a function of zero average, and let $0<\mu<1/4$.  
Then 
 $$
\sum_{\mu\in(0,1/4)\cap \Omega_\Gamma}|\scrD_\mu^-(|f \star \sigma_H|^2)| \ll_{\Gamma} (1+|a|^{-1}) 
H^{-1} S_{2, (7 - \nu)/2}(f)^2\,.
$$  
\end{proposition}
We 
use the spectral decomposition to get
\begin{align*}
 \scrD_\mu^-(|f \star \sigma_H|^2)=\sum_{\beta_1, \beta_2 \in \Omega_{\Gamma}}\scrD_\mu^-(f_{\beta_1} \star 
\sigma_H\overline{f_{\beta_2}\star \sigma_H}).
\end{align*}
In order to estimate $\scrD_\mu^-(f_{\beta_1} \star \sigma_H\overline{f_{\beta_2}\star \sigma_H})$, 
we start by defining a bi-sesquilinear functional $\dtwo$ on 
$V_{\beta_1}^\infty\times V_{\beta_2}^\infty$, given by:
\begin{equation}
 \notag
 \dtwo(f_1,f_2)=\scrD_\mu^-(f_1\bar{f_2}).
\end{equation}
For any $b\in B$, $\dtwo$ satisfies 
\begin{align}
\notag
\dtwo((b\cdot f_1), (b\cdot f_2))=\chi_\mu(b)\dtwo(f_1, f_2)
\end{align}
where
\begin{equation}
 \chi_\mu(n(x)a(t))=e^{(-1+\nu)t/2}.\notag
\end{equation}

Let $\eone$ be the space of bi-sesquilinear functionals $\scrD$ on 
$V_{\beta_1}^\infty 
\times V_{\beta_2}^\infty$ satisfying 
\begin{align}
 \lambda\scrD(f,g)&=\scrD(\lambda f,g)=\scrD(f,\overline{\lambda}g),\: \scrD(b\cdot f,b\cdot g 
)=\chi_\mu(b)\scrD(f,g),\label{eq: D def}
\end{align}
for any $b$ in the Borel subgroup of $\PSL(2, \RR)$, and $\lambda\in \CC$. We begin by finding the dimension of 
$\eone$:

\begin{lemma}
\label{lemm:multiplicity two}
For $\beta_1,\beta_2\in\Omega_\Gamma$, $\eone$ is a two dimensional space. Moreover, if for $j=1,2$, $\phi_j: 
V_{\beta_j}\rightarrow K_{\beta_j} $ is the equivalence map, then the space $\eone$ is 
spanned by the following two linearly independent functionals:
 \begin{align}\label{eq: B def1} 
 \Bone (f_{\beta_1}, f_{\beta_2}) & :=\int_{0}^\infty |x|^{-(1+\nu)/2} (\phi_1f_{\beta_1})(x) 
\overline{(\phi_2f_{\beta_2})(x)}dx,
\\
 \bone (f_{\beta_1}, f_{\beta_2}) & := \int_{-\infty}^0 |x|^{-(1+\nu)/2}  (\phi_1f_{\beta_1})(x) 
\overline{(\phi_2f_{\beta_2})(x)}dx.\label{eq:b def1}
\end{align}

\end{lemma}
\begin{proof}
We start by considering the space $\etwo $ of bi-sesquilinear 
 functionals on the line models $H_{\beta_1}^\infty\times H_{\beta_2}^\infty 
$ satisfying \eqref{eq: D def}, where $\nu_i=\sqrt{1-4\beta_i} $, 
if $\beta_i>0 $, and $\nu_i=n_i - 1$ for $n_i \in \ZZ^+$, if $\beta_i=-n_i^2+2n_i$. 


We follow the recipe in \cite[VII, 3.1]{GGV} 
to prove that the space $\etwo$ is at most two dimensional. 
This follows in a rather straight forward manner. 
Therefore, we only provide an 
outline of the argument here. 
By definition, $\etwo$ is the space of functionals 
$\scrD\in\scrE'(H_{\beta_1}^\infty, H_{\beta_2}^\infty) $ satisfying
\begin{align*}
 \scrD(n(t)\cdot\psi_1, n(t)\cdot\psi_2)&=\scrD(\psi_1,\psi_2)\\
 \scrD(a(t)\cdot\psi_1, a(t)\cdot\psi_2)&=e^{(-1+\nu)t/2}\scrD(\psi_1,\psi_2).
\end{align*}
Using \eqref{eq:line model action}, these conditions translate to
\begin{align}
 \scrD(\psi_1(x-t),\psi_2(x-t))&=\scrD(\psi_1,\psi_2)\label{eq:translation}\\
 \scrD(\psi_1(e^{-t}x),
\psi_2(e^{-t}x))&=e^{(1+\nu_1+\nu_2+\nu) t/2}\scrD(\psi_1,\psi_2).\label{eq:multiplication}
\end{align}
By \cite[VII, 3.1, (2)]{GGV}, condition \eqref{eq:translation} implies that
$
 \scrD(\psi_1,\psi_2)=\scrD_0(\omega),
$
where $\scrD_0$ is a one-dimensional distribution, and $\omega$ is the convolution 
$$
\omega(x)=\int \psi_1(x_1)\overline{\psi_2(x+x_1)}dx_1. 
$$
Condition \eqref{eq:multiplication} now implies that 
$ \scrD_0(\omega)=e^{-(1+\nu_1+\nu_2+\nu)t/2}\scrD_0(\omega(e^{-t}x)),$ or equivalently,
\begin{align*}
 \scrD_0(\omega)=|\alpha|^{(1+\nu_1+\nu_2+\nu)/2}\scrD_0(\omega(\alpha x)).
\end{align*}

This shows that $\scrD_0$ is a generalized homogeneous function of degree $(-1+\nu_1+\nu_2+\nu)/2$. 
The space of homogeneous
functionals on $\RR$  have been characterized completely 
and it is at most two dimensional. See \cite[VII, 3.1, (6)
and (7)]{GGV} for more details. 

Moreover, note that any functional in $\eone$ can be realized as a functional in $\etwo$, via the equivalence of 
$V_{\beta}$ with $K_\beta$ for any $\beta\in \Omega_\Gamma$, and using the map $\phi$ in (\ref{eq: Kirillov map}) 
between the Line model and the Kirillov model. This implies that the space $\eone$ is also at most two dimensional. 
However, it can be easily checked that the functionals $\Bone$ and $\bone$ defined by (\ref{eq: B def1}) and (\ref{eq:b 
def1}) are in $\eone$. Furthermore, since the space of Kirillov model $K_{\beta}$, for any $\beta\in \Omega_\Gamma$, 
contains the space of smooth compactly supported functions $\C^\infty_c(\RR\setminus\{0\})$, it can be easily deduced 
that $\Bone$ and $\bone$ are linearly independent, thus proving the lemma. 
\end{proof}
We start by proving the bound:

\begin{lemma}\label{lemm:functionals}
Let $0<\mu<1/4$. Then $\Bone$ and $\bone$ belong 
to 
$\eone$ and satisfy
$$
|\Bone(f_{\beta_1},f_{\beta_2})|  + |\bone(f_{\beta_1},f_{\beta_2})| \leq 
S_{2,0}(X^{(1 - \nu)/2} f_{\beta_1})S_{2,0}(X^{(1 - \nu)/2}f_{\beta_2})\,.
$$
\end{lemma}
\begin{proof}
Recall that for $j=1,2$, $X(\phi_jf)=ix\phi_jf$, 
then $iX$ is a self-adjoint operator with spectrum $\RR$.  
The spectral theorem then shows 
$X^s \phi_j f = e^{i \pi s/2} x^s \phi_j f$ for any $s \geq 0$.  
Therefore we have 
\begin{align*}
&|\Bone(f_{\beta_1}\otimes f_{\beta_2})|  + |\bone(f_{\beta_1} , f_{\beta_2})| \\
&\leq \int_{\RR / \{0\}} |x|^{-(1+\nu)/2} | (\phi_1f_{\beta_1})(x) \overline{  (\phi_2f_{\beta_2})(x) }| dx\\
 &\leq  \left(\int_{\RR / \{0\}}|x|^{-(1+\nu)/2} | (\phi_1f_{\beta_1})(x) |^2 dx\right)^{1/2}\left(\int_{\RR / 
\{0\}}|x|^{-(1+\nu)/2} | (\phi_2f_{\beta_2})(x) |^2 dx\right)^{1/2}\\
& =  \left(\int_{\RR / \{0\}}|x^{(1 - \nu)/2} (\phi_1f_{\beta_1})(x)|^2 \frac{dx}{|x|}\right)^{1/2}\left(\int_{\RR / 
\{0\}} |x^{(1 - \nu)/2} (\phi_2f_{\beta_2})(x)|^2 \frac{dx}{|x|}\right)^{1/2} \\ 
& = S_{2,0}(X^{(1 - \nu)/2} \phi_1 f_{\beta_1})S_{2,0}(X^{(1 - \nu)/2}\phi_2f_{\beta_2})\,.
 \end{align*}
\end{proof}


Since $\dtwo$ belongs to $\eone$, it can be written as a linear combination of $\Bone$ and $\bone $ as follows. 
 \begin{lemma}\label{lemm:constants}
There are constants $\cone, \ctwo \in \mathbb C$ such that 
$$
|\cone| + |\ctwo| \ll (1 + |\beta_1|)^{3/2} (1 + |\beta_2|)^{3/2}
$$
and 
$$
 \dtwo = \cone \Bone + \ctwo \bone.
$$
on $V_{\beta_1}^2\times V_{\beta_2}^2$. 
\end{lemma}
\begin{proof}
Let $(f, g) \in (V_{\beta_1}, V_{\beta_2})$ be such that 
$\phi_1(f), \phi_2(g) \in C_c^\infty((1, 2))$ 
are non-negative valued functions taking the value $1$ on the interval 
$(5/4,7/4)$.  
By a direct computation in the Kirillov model, 
we may further choose $f, g$ such that for any integer $k\geq 0$,
\begin{equation}\notag
\begin{aligned}
S_{2,k}(f) & \ll (1 + |\beta_1|)^k \\
S_{2,k}(g) & \ll (1 + |\beta_2|)^k \,.
\end{aligned}
\end{equation} 
Using interpolation, the above bounds hold for any $k\in \RR_+ $. 
Since $(f, g) \in (V_{\beta_1}^\infty, V_{\beta_2}^\infty)$ and 
$\bone$ vanishes on $(f,g)$, we get
\begin{align*}
 \scrD_\mu^{-}(f\overline g)&= \cone \Bone(f,g).
\end{align*}
Moreover,
\begin{align*}
  \Bone (f, g)=\int_{1}^2 |x|^{-(1+\nu)/2} (\phi_1f)(x) 
\overline{(\phi_2g)(x)}dx\gg 1.
\end{align*}
Since the $W^{-3/2-\ve}(M) $
norm of $\scrD_\mu^{-}$ is bounded by a constant $C > 0$, using the fact that $W^s(M)$ is 
a Banach algebra for any $s>3/2$, we get 
\begin{multline*}
 |\cone| \ll |\scrD_\mu^{-}(f\overline g)| \ll S_{2, 3/2+\ve}(f \overline g) \ll S_{2,3/2+\ve}(f) S_{2,3/2+\ve}(g)\\ \ll 
(1+|\beta_1|)^{3/2+\ve}(1+|\beta|)^{3/2+\ve} \,,
\end{multline*}
thus giving the bound on $\cone$. 
A similar treatment implies the result for $\ctwo$.
\end{proof}
We now use proposition \ref{prop:L^2_est} to obtain the following bound for $\Bone$ and 
$\bone$.
\begin{lemma}\label{lemm:Bone_decay}
\label{lem: bone bound}  
For $\mu \in \Omega_\Gamma \cap (0, 1/4)$ and $i=1,2$, we have
$$
|B_{\mu,\beta_1,\beta_2}^{i}(f_{\beta_1}\star \sigma_H, f_{\beta_2}\star \sigma_H)|
\ll (1+|a|^{-1}) H^{-1} S_{2, (3 - \nu)/2}(f_{\beta_1}) S_{2, (3 - \nu)/2}(f_{\beta_2})\,.
$$
 $i = 1, 2$.
\end{lemma}
\begin{proof} 
We only deal with bounding $\Bone$ here, 
the other case is analogous. 
We begin by applying lemma 
\ref{lemm:functionals} to get
\begin{multline*}
|\Bone(f_{\beta_1}\star \sigma_H, f_{\beta_2}\star 
\sigma_H)| 
\ll S_{2,0}((X^{(1 - \nu)/2} f)_{\beta_1}\star\sigma_H))S_{2,0} ((X^{(1 - \nu)/2} f)_{\beta_2}\star\sigma_H )\,,
\end{multline*}
since $X$ commutes with $\star\,\sigma_H$. 
We now invoke the estimate in proposition \ref{prop:L^2_est} to get
\begin{align*}
&|\Bone(f_{\beta_1}\star \sigma_H, f_{\beta_2}\star 
\sigma_H)|\ll (1+|a|^{-1})H^{-1}
S_{2, (3 - \nu)/2}(f_{\beta_1}) S_{2, (3 - \nu)/2}(f_{\beta_2})\,.
\end{align*}
\end{proof}

Finally, we use these functionals to estimate $\mathcal D_\mu^-(|f \star \sigma_H|^2)$ and prove proposition 
\ref{prop:distribution}.  

\begin{proof}[Proof of Proposition \ref{prop:distribution}]
Using Lemma \ref{lemm:Bone_decay} and Lemma \ref{lemm:constants}, 
we get 
\begin{align}\label{equa:D_mu on betas}
&|\mathcal D_\mu^-(|f \star \sigma_H|^2)|\\
& = \left|\mathcal D_\mu^-\left(\sum_{\beta_1,\beta_2\in \Omega_\Gamma}
(f\star \sigma_H)_{\beta_1} (\overline{f\star \sigma_H})_{\beta_2}\right)\right|\leq \sum_{\beta_1,\beta_2\in \Omega_\Gamma}\left|
\dtwo\left(f_{\beta_1}\star \sigma_H,f_{\beta_2}\star 
\sigma_H\right)\right|\notag\\
 &\ll  (1+|a|^{-1})H^{-1} \sum_{\beta_1,\beta_2\in \Omega_\Gamma}(1 + |\beta_1|)^{3/2+\ve} S_{2,(3 - \nu)/2}(f_{\beta_1})(1 
+ |\beta_2|)^{3/2+\ve}S_{2,(3 - \nu)/2}(f_{\beta_2}) \notag \\
 & \ll (1+|a|^{-1})H^{-1} \sum_{\beta_1,\beta_2\in \Omega_\Gamma}S_{2,(9 - \nu)/2+2\ve}(f_{\beta_1}) S_{2,(9 - 
\nu)/2+2\ve}(f_{\beta_2}),\notag
 \end{align} 
using the fact that $(1+|\beta_j|)^kS_{2,s}(f_{\beta_j})\ll S_{2,s+2k}(f_{\beta_j}) $, for $j=1,2$, and for any 
$k,s\in \RR_+$. Now using the Cauchy-Schwarz inequality and the Plancherel formula, 
we get 
\begin{align}
&\ll \#\{\Omega_\Gamma\cap (0,1/4)\}(1+|a|^{-1}) H^{-1} 
\left(\sum_{ 
\beta\in \Omega_\Gamma}S_{2, (9 - \nu)/2+2\ve}(f_{\beta})\right)^2 \notag \\
&\ll \#\{\Omega_\Gamma\cap (0,1/4)\}(1+|a|^{-1}) H^{-1} 
\left(\sum_{ 
\beta\in \Omega_\Gamma}(1+|\beta|)^{-1/2-\ve/2}S_{2, (11 - \nu)/2+3\ve}(f_{\beta})\right)^2\notag \\
&\ll (1+|a|^{-1}) H^{-1}\#\{\Omega_\Gamma\cap (0,1/4)\}(\sum_{\beta\in\Omega_\Gamma}(1+|\beta|)^{-1-\ve})S_{2, 
(11 - \nu)/2+3\ve}(f)^2\,.\notag
\end{align}  
The Weyl's law for the distribution of eigenvalues \cite[Lemma 2.28]{DH} implies that $\#\{\beta\in\Omega_\Gamma, |\beta|\leq T_0\}\ll T_0 $, 
thus implying
$\sum_{\beta\in\Omega_\Gamma}|\beta|^{-1-\ve}<\infty$. Further adding over all $\mu\in(0,1/4)\cap\Omega_\Gamma$ 
establishes the proposition.
\end{proof}

\subsection{Proof of theorem \ref{theo:main theorem}}
We start by bounding the $W^s$ norm of $|f\star \sigma_H|^2 $.
\begin{lemma}\label{lem:Sobolev_f-star}
Let $f \in \C^\infty(M)$, $H \geq 1$ and $\ve > 0$. Let $\eta>0$ be such that for any $k\in\NN$,
$$S_{\infty,k}(f\star\sigma_H)\ll C_{a,k} H^{2k-\eta}S_{2,k+11/2+\ve}(f),$$
where $C_{a,k}$ is a constant depending on $a$ and $k$, satisfying $C_{a,k}\leq C_{a,k+1}$.
Then for any $s\in\RR_+$, we have 
$$
S_{2, s}(|f\star\sigma_H|^2) \ll_\ve (1+|a|^{-1/2})C_{a,\lceil s\rceil} H^{2s-1/2-\eta +\ve} S_{2, s+11/2+\ve}(f)^2,
$$
\end{lemma}
where $\lceil s\rceil$ denotes the nearest integer greater than or equal to $s$.
\begin{proof}

Let $s \geq 0$ be an integer. We can easily see that after using proposition \ref{prop:L^2_est} we get
\begin{align}
S_{2, s}(|f \star \sigma_H|^2)^2&\ll \sum_{j = 0}^s S_{\infty,j}(f\star\sigma_H)^2 S_{2,s-j}(f\star\sigma_H)^2\ll 
C_{a,s}(1+|a|) H^{4s-1-2\eta}S_{2,s+11/2+\ve}(f)^4. \notag
\end{align}
The bound can then be extended 
for any $s\in\RR_+$, after using interpolation.
\end{proof}
\begin{proof}[Proof of theorem \ref{theo:main theorem}]
 We begin by observing that 
\eqref{eq:inftybound} implies that for any $x_0\in M$, we have
\begin{align}\label{eq:infty2}
|f \star \sigma_T(x_0)| \ll 
\sum_{j=0}^{s}\sum_{B\in\scrO_j}\nu_T(|Bf\star\sigma_H|^2)^{1/2}T^{2j}+HT^{2s-1}S_{2,s}(f).
\end{align}
We are now ready to finish the first part of theorem \ref{theo:main theorem}, namely the bound (\ref{eq:1/8}). Using 
\eqref{eq:infty 2sbound}, and Sobolev embedding, it is easy to establish the bound
\begin{align}
 S_{\infty,k}(f\star\sigma_H)\leq C_k H^{2k}S_{\infty,s}(f)\ll C'_k 
H^{2k}S_{2,s+3/2+\ve}(f),\label{eq:Sobolev_f-star}
\end{align}
where $C'_k$ is an increasing sequence. Thus, the hypothesis of lemma \ref{lem:Sobolev_f-star} is valid with 
$\eta=\eta_0:=0$. This implies the bound
\begin{align*}
 S_{2, s}(|f\star\sigma_H|^2) \ll_\ve (1+|a|^{-1})C_{k}' H^{2s-1/2-\eta_0 +\ve} S_{2, s+11/2+\ve}(f)^2.
\end{align*}
Keeping the explicit dependence on $\eta$ will be useful in proving the later part of the proof.
  By substituting the above bounds, along with the ones from proposition \ref{prop:L^2_est}, 
 and  proposition \ref{prop:distribution} , into lemma \ref{lemm:nu_T},
for any $\ve > 0$, we have,
$$
\nu_T(|Bf\star\sigma_H|^2) \ll_{\ve,\Gamma} (1+|a|^{-1})\left(\frac{1}{ H} + 
\frac{H^{1+\ve}}{T^{1/2-\ve}}+ \frac{H^{7/2-\eta_0+ \ve}}{T}\right) (S_{2, 11/2+\ve}(Bf))^2,
$$

Applying these bounds to (\ref{eq:infty2}), we get
$$
S_{\infty,s}(f \star \sigma_T) 
\ll_{\ve,\Gamma} (1+|a|^{-1/2}) T^{2s}\left(H^{-1/2} + \frac{H^{7/4-\eta_0/2 + \ve}}{\sqrt T}\right) S_{2,s+ 
11/2+\ve}(f).
$$
Optimizing, we set $H = T^{2/(9-2\eta_0)-\ve}$, which yields
\begin{align}
\label{eq:1/8bound}
S_{\infty,s}(f \star \sigma_T) \ll_{\ve,\Gamma}(1+|a|^{-1/2}) T^{2s-1/(9-2\eta_0) +\ve} S_{2, s+11/2+\ve}(f),
\end{align}
thus proving the bound \eqref{eq:1/8}, upon recalling that $\eta_0=0$. The explicit dependence of $\psi$ and $T$ in \eqref{eq:1/8bound} 
will be crucial in the application to the sparse equidistribution. 

Henceforth, we assume that $a\neq 0$ is fixed. The validity of the bound (\ref{eq:1/8-}) implies that the hypothesis of 
lemma \ref{lem:Sobolev_f-star} holds with $\eta=\eta_1=1/9-\ve$, and $C_{a,k}\ll_k(1+|a|^{-1/2}) $. Now, the process of 
obtaining \eqref{eq:1/8bound} can be bootstrapped to obtain that for any $j\in\NN$, we have
\begin{align}
\label{eq:1/8bound1}
S_{\infty,s}(f \star \sigma_T) \ll_{\ve,\Gamma,s,j}(1+|a|^{-1/2})^{j+1} T^{2s-1/(9-2\eta_j) +\ve} S_{2, s+11/2+\ve}(f),
\end{align}
where $\eta_0=0$, and the sequence $\eta_j$ satisfying $\eta_{j+1}:=\frac{1}{9-2\eta_{j}} $. It can be easily seen 
that the sequence $\eta_j $ converges to $(9-\sqrt{73})/4$, which is a solution to the quadratic equation $2y^2-9y+1$, 
thus proving \eqref{eq:1/8-}, and the theorem.
\end{proof}

\section{Proof of theorem \ref{thm:main theorem 2} and Theorem \ref{theo:equi-time-changemaps}}
\subsection{Proof of theorem \ref{thm:main theorem 2}}
We follow a variant of the recipe in the proof of \cite[Theorem 3.1]{V}. 
We start by proving effective equidistribution for arithmetic progressions:
\begin{lemma}
\label{lem:Kinfty}
 Let $x_0\in M$, and let $f\in \C^\infty(M) $ be such that $\int_M f(g)dg=0 $. 
 Let $b,b_1,\ve$ be positive numbers 
satisfying $b+\ve<\frac{(1-2\alpha_0)^2}{8(3-2\alpha_0)}$ and $b_1+\ve<1/9$. 
Then 
\begin{align*}
 \left|\sum_{1\leq j\leq K^{r-1}}f(x_0n(Kj))\right|\ll_{f,\ve} K^{r-1-\ve},
\end{align*}
for any $r\geq 3/(b+2b_1) $.
\end{lemma}
\begin{proof}
Let $g_\delta=\max(\delta^{-2}(\delta-|t| ),0) $ be a function on $\RR$. 
Using the Poisson summation formula, it can be easily seen that
\begin{align*}
 \sum_{j\in\ZZ}g_\delta(t+jK)=\sum_{k\in \ZZ}\exp(2\pi i K^{-1} kt)a_k,
\end{align*}
where $a_\lambda=K^{-1}\int_\RR \exp(-2\pi i \lambda t)g_\delta(t)dt $. 
It can be easily seen that 
$|a_k|\leq K^{-1} $ 
and 
$\sum_k |a_k|\ll \delta^{-1} $. Note that 
\begin{align}
 \int_0^T \left( \sum_{j\in \ZZ}g_\delta(t+Kj)\right) f(x_0n(t))dt=\sum_{k\in \ZZ}a_k\int_0^T\exp(2\pi i 
K^{-1}kt)f(x_0n(t))dt. \label{eq:aksum}
\end{align}
We use the bound in lemma \ref{lem:ven} to bound the 
integrals on the right hand side of \eqref{eq:aksum} when 
$|k|<K_0$, say, and use the bound of theorem \ref{theo:main theorem} to bound the rest. 
We then have
\begin{align}
 \left|\int_0^T \left( \sum_{j\in \ZZ}g_\delta(t+Kj)\right) f(x_0n(t))dt\right|\ll_{\Gamma,f,\ve} K_0K^{-1}
T^{1-b-\ve}+(K/K_0)^{1/2}T^{1-b_1-\ve}\delta^{-1}.\label{eq:aksplit}
\end{align}
Moreover, since $g_\delta$ is supported in a $\delta$ neighborhood of $0$, and it has integral 
$1$, we can easily 
deduce that
\begin{align*}
 \left|\int_0^T \left( \sum_{j\in \ZZ}g_\delta(t+Kj)\right) f(x_0n(t))dt-\sum_{\substack{j\in \ZZ\\1\leq Kj\leq 
T}}f(x_0n(Kj))\right|\ll_f (1+TK^{-1}\delta).
\end{align*}
Combining with \eqref{eq:aksplit}, we get
\begin{align}\label{equa:combine-split}
  \left|\sum_{\substack{j\in \ZZ\\1\leq Kj\leq 
T}}f(x_0n(Kj))\right|\ll_{\Gamma,f,\ve} 1+TK^{-1}\delta+ K_0K^{-1} 
T^{1-b-\ve}+(K/K_0)^{1/2}T^{1-b_1-\ve}\delta^{-1}.
\end{align}
Choose $T=K^r$, $K_0=K^{rb}$, and $\delta=K^{-\ve}$ to get
\begin{align*}
  \left|\sum_{1\leq j\leq 
K^{r-1}}f(x_0n(Kj))\right|\ll_{\Gamma,f,\ve} 1+K^{r-1-\ve}+
K^{r-1-r\ve}+K^{r-1+(3/2-r(b/2+b_1))-(r-1)\ve}.
\end{align*}
The lemma now follows upon choosing $r\geq 3/(b+2b_1)$.

\end{proof}
\begin{proof}[Proof of theorem \ref{thm:main theorem 2}]
By possibly subtracting $f$ by a constant, we may assume that $f$ is a function of zero average. The implied constants appearing in the proof here are independent of $N$, but may depend on $f,\Gamma$, and a parameter $\ve$ chosen in due course. 

The theorem follows easily from lemma \ref{lem:Kinfty} 
after approximating the sequence $\{j^{1+\gamma}:0\leq j\leq 
N\} $ by a union of arithmetic progressions. 
In particular, let $N_0\in \NN$ be a large number. 
For a small $t$ we have
\begin{equation}
 \label{eq:N_0sum}
(N_0+t)^{1+\gamma}=N_0^{1+\gamma}(1+t/N_0)^{1+\gamma}=N_0^{1+\gamma}+(1+\gamma)tN_0^\gamma+O(t^2N_0^{\gamma-1}). 
\end{equation}

This is a good approximation for $t\ll N_0^{(1-\gamma)/2-\ve} $, where $\ve$ is a small positive number. 
For $N_1=N^{1-\ve}$, we write
$\{j^{\gamma+1}:0\leq j\leq N\}=\{j^{\gamma+1}:0\leq j\leq N_1-1\}\cup \{j^{\gamma+1}:N_1\leq j\leq N\} $. 
The second set can be decomposed into a 
disjoint union of $L-1$ sets of the form 
$$
\cup_{k=1}^{L-1}\{N_{k}^{1+\gamma},(N_k+1)^{1+\gamma},...,(N_k+N_{k}^{(1-\gamma)/2-\ve})^{1+\gamma}\}\cup\{N_{L}^{1+\gamma}
,...,N^{1+\gamma}\} ,
$$ 
where for any $k\geq 1$, $N_{k+1}=N_k+N_{k}^{(1-\gamma)/2-\ve}+1$, and $N_L\leq N\leq N_{L+1}$. Clearly, the terms in 
the tail $N^{-1}\sum_{j=N_L}^N f(x_0 n(j^{1+\gamma})) $ can be bound appropriately. Since each of these above sets have 
$N_k^{(1-\gamma)/2-\ve}$ elements, we have
$$\sum_{k=1}^{L-1} N_k^{(1-\gamma)/2-\ve}\ll N.$$

For each $1\leq k< L$, we can now 
apply the lemma \ref{lem:Kinfty} along with 
\eqref{eq:N_0sum} as long as $N_k^{(1-\gamma)/2-\ve}>N_k^{\gamma(r_0-1)}$, 
where $r_0=3/(b+2b_1)$. 
If $\frac{1-\gamma}{2\gamma}>r_0-1$, 
equivalently if $\gamma<1/(2r_0-1)=1/(6/(b+2b_1)-1)$, a suitable value of $\ve>0$ can be chosen so that the above condition as well as the hypothesis of lemma \ref{lem:Kinfty} hold. The theorem now 
follows easily from the following estimates
\begin{multline*}
 \frac{1}{N}\left|\sum_{k=1}^{L-1}\sum_{j=1}^{N_k^{(1-\gamma)/2-\ve}}f(x_0n((N_k+j)^{1+\gamma}))\right|\ll
\frac{1}{N}\sum_{k=1}^{L-1}\left|\sum_{j=1}^{N_k^{(1-\gamma)/2-\ve}}f(x_0n(N_k+N_k^\gamma j))\right|\\
\ll
\frac{1}{N}\sum_{k=1}^{L-1}N_k^{(1-\gamma)/2-\ve-\gamma\ve}\ll N_1^{-\gamma\ve}\ll N^{-(1-\ve)\gamma\ve}.
\end{multline*}

\end{proof}
\subsection{Proof of theorem \ref{theo:equi-time-changemaps}}
Let 
$
\L_0^2(M) := \left\{f \in L^2(M) :  \int_M f(g)\, d_\rho g = 0\right\}\, ,
$
where $\rho\in W^6(M)$ is a positive function. 
Using the commutation relations 
$$
[X_\rho, Y] = (\frac{Y \rho}{\rho} - 1) X_\rho\,, \ [X_\rho, Z] = \frac{Y}{\rho} + \frac{Z\rho}{\rho} X_\rho\,,
$$ 
and solving a system of O.D.E's, the tangent flow $\{D n^\rho_t\}$ on $TM$ is computed 
in \cite[Lemma~1]{FU}.  
It follows that there are continuous functions 
$
y_{X_\rho}, z_{X_\rho}, z_Y: \RR \to \RR
$ 
satisfying 
$$
|y_{X_\rho}(t)| + |z_Y(t)|\ll_\rho |t| \,,\,
|z_{X_\rho}(t)|\ll_\rho |t|^2\,,
$$
such that for every $f \in C^\infty(M)$, 
\begin{align}\label{equa:time-change-tangentflow}
Y(n^\rho(t)\cdot f) & = \left(y_{X_\rho}(t) n^\rho(t)\cdot X_\rho + n^\rho(t)\cdot Y\right)(f)\,, \notag \\
Z(n^\rho(t)\cdot f) & = \left(z_{X_\rho}(t) n^\rho(t)\cdot X_\rho + z_Y(t) n^\rho(t)\cdot Y + n^\rho(t)\cdot Z\right)(f)\,, \notag \\
X_\rho (n^\rho(t)\cdot f)&  = n^\rho(t)\cdot X_\rho f \,.
\end{align}

The main step in proving theorem~\ref{theo:equi-time-changemaps} 
is the following lemma.  
\begin{lemma}\label{lemm:twisted_int-timechange}
For all $f \in C^\infty(M) \cap \L_0^2(M)$, 
and all $T \geq 1$, and $\ve > 0$, 
$$
S_{\infty,0}(f\star\sigma_T^\rho) \ll_\rho T^{-(1 - \alpha_0)^2/(100 - 4\alpha_0)+\ve} S_{2, 15/2}(f)\,.,
$$
where
$$
f\star\sigma_T^\rho(f)(x) := \frac{1}{T} \int_0^T \psi(t) f(xn^\rho(t))\, dt\,. 
$$
\end{lemma}
\begin{proof}
The proof will follow by combining the argument in \cite[Lemma 3.1]{V}  
with results on the quantitative equidistribution 
and quantitative mixing of $\{n^\rho(t)\}_{t\in \RR}$ in \cite[Theorems 2,3]{FU}.    

We provide slightly weaker versions of these results here.  Recall that $\alpha_0 = \sqrt{1 - 4\lambda_1} \in (0, 1]$, where $\lambda_1$ is 
the spectral gap of the Laplace-Beltrami operator on $M$.    
\begin{lemma}[Theorems 2 and 3 in \cite{FU}]\label{lemm:time-changes-equi}
\hfill
\newline

$\bullet$ For any $r > 3$, $T \geq 1$, $x_0\in M$, and $f\in W^r(M)$, we have
\begin{align}\label{equa:time-changes-equi}
|\int_0^T f(x_0 n^\rho(t))\, dt - \int_M f(g)\,d_\rho g| & \ll_{r, \rho} T^{-\frac{1 - \alpha_0}{2}} (1 + \log T) S_{2, r}(f ) \notag 
\end{align}

$\bullet$ For any $r > 11/2$, $(x, t) \in M \times \RR_{\geq 1}$, and for any 
$f \in W^r(M)\cap \L_0^2(M)$ and  any $g \in W^r(M)$, 
\begin{equation}\label{equa:time-changes-mixing}
|\langle n^\rho(t) f, g\rangle_{L^2(M, \vol_\rho)}| 
\ll_{r, \rho} S_{2, r}(f) S_{2, r}(g) t^{- \frac{1 - \alpha_0}{2}} (1 + \log t)\,. \notag
\end{equation}
\end{lemma}
 
Let 
$
\nu_{T}^\rho(f)  := \frac{1}{T} \int_0^T f(x_0n^\rho(t))\, dt\,.
$
We can easily derive a bound analogous to \eqref{eq:inftybound}, and \cite[(3.5)]{V}: 
\begin{align}\label{equa:Venkatesh-argument}
|S_{\infty,0}(f\star \sigma_T^\rho)| & \ll \frac{H}{T} S_{\infty, 0}(f) + \sqrt{\nu_T^\rho(|f\star\sigma_H^\rho|^2)} \notag \\
& = \frac{H}{T} S_{\infty, 0}(f) + \left(\frac{1}{H^2} \int_{(h_1, h_2) \in [0, H]^2} |\nu_T(n^\rho(h_1) f \cdot \overline{n^\rho(h_2) f})| dh_1 dh_2\right)^{1/2} \notag \\
& \ll_{\rho} \frac{H}{T} S_{\infty, 0}(f) + \left(\frac{1}{H^2} \int_{(h_1, h_2) \in [0, H]^2} |\langle n^\rho(h_1 - h_2) f, f\rangle| dh_1 dh_2\right)^{1/2} \notag \\ 
& \ \ \ \ \ \ \ \ \ \ \ \ \ \ \ \ \ \ \ + \left(T^{-\frac{1 - \alpha_0}{2}} (1 + \log T)\sup_{(h_1, h_2) \in [0, H]^2} S_{2, 6}( n^\rho(h_1) f \cdot \overline{n^\rho(h_2) f)}\right)^{1/2} \,.
\end{align}
Using (\ref{equa:time-change-tangentflow}), we also can derive bounds analogous to \eqref{eq:Sobolev_f-star} 
$$
S_{2, 6}(n^\rho(h_1)\cdot f \,\,\overline{n^\rho(h_2)\cdot f)} \ll_\rho (1 + |h_1| + |h_2|)^{12} S_{2, 15/2}(f)^2\,.   
$$
Applying the 
Sobolev embedding theorem, along with the mixing statement in 
lemma~\ref{lemm:time-changes-equi}, to
(\ref{equa:Venkatesh-argument}) gives that for all $\ve > 0$, we have
\begin{equation}\label{equa:time-changeoptimize}
(\ref{equa:Venkatesh-argument}) \ll_{\rho} (\frac{H}{T} + H^{(\alpha_0 - 1)/4 +\ve} + \frac{H^{6}}{T^{\frac{1 - \alpha_0}{4}- \ve}}) S_{2, 15/2}(f).
\end{equation}

Optimizing, set $H = T^{(1 - \alpha_0)/(25 - \alpha_0)-\ve},$ 
and we get  
$$
(\ref{equa:time-changeoptimize}) \ll_\rho T^{-(1 - \alpha_0)^2/(100 - 4\alpha_0)+\ve} S_{2, 15/2}(f)\,.
$$
This concludes the proof of 
lemma~\ref{lemm:twisted_int-timechange}.  
\end{proof}

\begin{proof}[Proof of Theorem \ref{theo:equi-time-changemaps}]
Let $f \in L_0^2(M)$. The method used to prove lemma \ref{lem:Kinfty} can be recycled here upon 
setting $K_0 = 1$, $b = b_1 = (1 - \alpha_0)^2/(100 - 4\alpha_0)$, 
and replacing $n(t)$ with $n^\rho(t)$.
A formula analogous to (\ref{equa:combine-split}) then gives 
$$
 \left|\sum_{\substack{j\in \ZZ\\1\leq j\leq T}} f(x_0n^\rho(K j)) \right| \ll_\rho (1+T K^{-1} \delta+ T^{1-b-\ve}\delta^{-1}) S_{2, 15/2}(f)\,.
$$
As before, let $T = K^r$ and $\delta = K^{-\ve}$ to get 
\begin{equation}\label{equa:time-change-horomap}
 \left|\sum_{\substack{j\in \ZZ\\1\leq j\leq K^r}} f(x_0n^\rho(K j)) \right| \ll_\rho (1+ K^{r-1 + \ve} + K^{r(1-b-\ve) - \ve}) S_{2, 15/2}(f)\,.
\end{equation}
Now using the notation in the proof of theorem~\ref{thm:main theorem 2}, 
we approximate the sequence $\{j^{1+\gamma}:0\leq j\leq 
N\} $ by a union of arithmetic progressions.  
We use (\ref{equa:time-change-horomap}) in the place of lemma~\ref{lem:Kinfty},
and conclude that theorem~\ref{theo:equi-time-changemaps} holds 
for any $\gamma < \frac{b}{2}.$ 
\end{proof}

\end{document}